\documentclass[a4paper,10pt]{article}
\usepackage{stmaryrd}
\usepackage{amsfonts}
\usepackage{bbm}
\usepackage{amscd}
\usepackage{mathrsfs}
\usepackage{latexsym,amssymb,amsmath,amscd,amscd,amsthm,amsxtra}
\usepackage[dvips]{graphicx}
\usepackage[utf8]{inputenc}
\usepackage[T1]{fontenc}
\usepackage{lmodern}
\usepackage{amssymb}
\usepackage[all]{xy}
\usepackage{nicefrac,mathtools,enumitem}
\usepackage[numbers,sort&compress]{natbib}
\usepackage{microtype}
\usepackage{xcolor}
\usepackage[all]{xy}
\newtheorem{thm}{Theorem}[section]
\newtheorem{lem}[thm]{Lemma}
\newtheorem{cor}[thm]{Corollary}
\newtheorem{pro}[thm]{Proposition}
\newtheorem{ex}[thm]{Example}

\newtheorem{defi}[thm]{Definition}

\newcommand {\emptycomment}[1]{}

\newcommand{\lon }{\,\rightarrow\,}
\newcommand{\be }{\begin{equation}}
\newcommand{\ee }{\end{equation}}

\newcommand{\g}{\frkg}

\newcommand{\huaB}{\mathcal{B}}

\newcommand{\huaC}{{\mathcal{C}}}

\newcommand{\huaH}{\mathcal{H}}

\newcommand{\huaO}{\mathcal{O}}

\newcommand{\huaZ}{\mathcal{Z}}

\newcommand{\frkg}{\mathfrak g}

\def\qed{\hfill ~\vrule height6pt width6pt depth0pt}

\newcommand{\br}[1]{   [ \cdot,    \cdot  ]_\frkg   }

\newcommand{\Id}{\rm{Id}}

\newcommand{\dM}{\mathrm{d}}

\newcommand{\Hom}{\mathrm{Hom}}

\newcommand{\Ad}{\mathrm{Ad}}

\newcommand{\gl}{\mathfrak {gl}}

\newcommand{\ad}{\mathrm{ad}}

\newcommand{\reg}{\mathrm{reg}}

\begin{document}
\title{
{ Representations and cohomologies of Hom-pre-Lie algebras
\thanks
 {
Research supported by NSFC (11471139) and NSF of Jilin Province (20170101050JC).
 }
} }
\author{\vspace{2mm}Shanshan Liu, Lina Song and  Rong Tang  \\
Department of Mathematics, Jilin University,\\\vspace{2mm}
 Changchun 130012, Jilin, China
\\\vspace{3mm}
Email:shanshan18@mails.jlu.edu.cn; ~songln@jlu.edu.cn; ~tangrong16@mails.jlu.edu.cn }

\date{}


\maketitle
\begin{abstract}
In this paper, first we study dual representations and tensor representations of   Hom-pre-Lie algebras. Then we develop the cohomology theory of Hom-pre-Lie algebras in term of the cohomology theory of Hom-Lie algebras. As applications, we study linear deformations of Hom-pre-Lie algebras, which are characterized by the second cohomology groups of Hom-pre-Lie algebras with the coefficients in the regular representation. The notion of a Nijenhuis operator on a Hom-pre-Lie algebra is introduced which can generate trivial linear deformations of a Hom-pre-Lie algebra. Finally, we introduce the notion  of   a Hessian structure   on a Hom-pre-Lie algebra, which is a symmetric nondegenerate 2-cocycle with the coefficient in the trivial representation. We also introduce the notion of an $\huaO$-operator on a Hom-pre-Lie algebra, by which we give an equivalent characterization of a Hessian structure.
\end{abstract}

\section{Introduction}

The representation theory of an algebraic
object is very important since it reveals some of its profound structures hidden underneath.  Furthermore, the cohomology theories of an algebraic
object  occupy a  central position since they can give invariants, e.g.   they can control deformations or extension problems. Representations and cohomology theories of various kinds of algebras have been developed with a great success. In particular, the representation theory, the cohomology theory and the deformation theory of Hom-Lie algebras were well studied in \cite{AEM,MS1,sheng3}. The purpose of this paper is to develop the representation theory and the cohomology theory of Hom-pre-Lie algebras.

In the study of $\sigma$-derivations of an associative algebra, Hartwig, Larsson
and Silvestrov introduced the notion of a Hom-Lie algebra in \cite{HLS}.
\begin{defi} A Hom-Lie algebra is a triple $(\g,[\cdot,\cdot],\alpha)$ consisting of a linear space $\g$, a skew-symmetric bilinear map $[\cdot,\cdot]:\wedge^2\g\longrightarrow \g$ and an algebra morphism $\alpha:\g\longrightarrow \g $, satisfying:
  \begin{equation}
[\alpha(x),[y,z]]+[\alpha(y),[z,x]]+[\alpha(z),[x,y]]=0,\quad
\forall~x,y,z\in \g.
\end{equation}
  \end{defi}
Some $q$-deformations of the Witt and the Virasoro algebras have the
structure of a Hom-Lie algebra. Because of their close relation
to discrete and deformed vector fields and differential calculus
\cite{HLS,LD1,LD2},   Hom-Lie algebras were widely studied. In particular, Hom-structures on semisimple Lie algebras were studied in \cite{Jin1,Jin2}; Geometrization of Hom-Lie algebras were studied in \cite{LGT,CLS};  Bialgebras for Hom-algebras were studied in \cite{Cai-Sheng,sheng1,Yao1,Yao3}; Recently,  Elchinger,  Lundengard,  Makhlouf and  Silvestrov extend the result in \cite{HLS} to the case of $(\sigma,\tau)$-derivations (\cite{ELMS}).

The notion of a Hom-pre-Lie algebra was introduced in \cite{MS2}.
 \begin{defi}
A Hom-pre-Lie algebra $(A,\cdot,\alpha)$ is a vector space $A$ equipped with a bilinear product $\cdot:A\otimes A\longrightarrow A$, and $\alpha\in \gl(A)$, such that for all $x,y,z\in A$, $\alpha(x \cdot y)=\alpha(x)\cdot \alpha(y)$ and the following equality is satisfied:
\begin{eqnarray}
(x\cdot y)\cdot \alpha(z)-\alpha(x)\cdot (y\cdot z)=(y\cdot x)\cdot \alpha(z)-\alpha(y)\cdot (x\cdot z).
\end{eqnarray}
\end{defi}Hom-pre-Lie algebras play important role in the construction of Hom-Lie 2-algebras (\cite{SC}). The geometrization of Hom-pre-Lie algebras was studied in \cite{Qing}, universal $\alpha$-central extensions of Hom-pre-Lie algebras were studied in \cite{sunbing} and the bialgebra theory of Hom-pre-Lie algebras was studied in \cite{QH}. There is a close relationship between Hom-pre-Lie algebras and Hom-Lie algebras: a Hom-pre-Lie algebra $(A,\cdot,\alpha)$ gives rise to a Hom-Lie algebra $(A,[\cdot,\cdot]_C,\alpha)$ via the commutator bracket, which is called the subadjacent Hom-Lie algebra and denoted by $A^C$. Furthermore, the map $L:A\longrightarrow\gl(A)$, defined by $L_xy=x\cdot y$ for all $x,y\in A$, gives rise to a representation of the subadjacent Hom-Lie algebra $A^C$ on $A$ with respect to $\alpha\in\gl(A)$.

Even though there are important applications of Hom-pre-Lie algebras,   cohomologies of Hom-pre-Lie algebras have not been well developed due to the difficulty. The purpose of this paper is to study  the cohomology theory of Hom-pre-Lie algebras and give its applications.

The notion of a representation of a Hom-pre-Lie algebra was introduced in \cite{QH} in the study of Hom-pre-Lie bialgebras (Hom-left-symmetric bialgebras). However, one need to add some complicated conditions to obtain the dual representation, which is  not convenient in applications. In this paper, we give the natural formula of a dual representation, which is nontrivial. Moreover, we show that there is well-defined tensor product of two representations of a Hom-pre-Lie algebra. This is the content of Section 3.

In Section 4, we give the cohomology theory of a Hom-pre-Lie algebra $(A,\cdot,\alpha)$ with the coefficient in a representation $(V;\rho,\mu)$. The main contribution is to define the coboundary operator $\partial:\Hom(\wedge^{n}A\otimes A,V)\longrightarrow\Hom(\wedge^{n+1}A\otimes A,V)$. To do that, first we show that the subadjacent Hom-Lie algebra $A^C$ represents on $\Hom(A,V)$ naturally and we obtain a cochain complex $(\oplus_n\Hom(\wedge^nA^C,\Hom(A,V)),\dM)$. Moreover, there holds: $\partial\circ\Phi=\Phi\circ \dM$, where $\Phi$ is the natural isomorphism from $\Hom(\wedge^nA^C,\Hom(A,V))$  to $\Hom(\wedge^nA\otimes A,V)$. Therefore, the operator $\partial$   is indeed a coboundary operator, i.e. satisfies $\partial\circ\partial=0.$ Consequently, we establish the cohomology theory for Hom-pre-Lie algebras.

In Section 5, we study linear deformations of a Hom-pre-Lie algebra using the cohomology with the coefficient in the regular representation defined in Section 4. We introduce the notion of a Nijenhuis operator on a Hom-pre-Lie algebra and show that it gives rise to a trivial deformation. Finally, we study the relation between linear deformations of a Hom-pre-Lie algebra and linear deformations of its subadjacent Hom-Lie algebra.

In Section 6, first we define   $\huaO$-operators on a Hom-pre-Lie algebra with respect to a representation and give its relation to  Nijenhuis operators. Then we define a Hessian structure on a  Hom-pre-Lie algebra  $(A,\cdot,\alpha)$ to be a nondegenerate 2-cocycle $\huaB\in Sym^2(A^*)$, and show that there is a one-to-one correspondence between some special  $\huaO$-operators on a Hom-pre-Lie algebra $(A,\cdot,\alpha)$ with respect to the representation $(A^*,(\alpha^{-1})^*, L^\star-R^\star,-R^\star)$ and Hessian structures on $(A,\cdot,\alpha)$.

\section{Preliminaries}
In this section, we recall representations, cohomologies and dual representations of Hom-Lie algebras. Throughout the paper, all the Hom-Lie algebras are regular.

\emptycomment{
\begin{defi} A morphism of Hom-Lie algebras $f:(\g,[\cdot,\cdot]_\g,\alpha_\mathfrak{g})\longrightarrow (\mathfrak{h},[\cdot,\cdot]_\mathfrak{h},\alpha_\mathfrak{h})$ is a linear map $f:\g\longrightarrow \mathfrak{h}$ such that
 \begin{eqnarray}
 f[x,y]_\g&=&[f(x),f(y)]_\mathfrak{h},\quad \forall~x,y\in \g,\\
 f\circ \alpha_\g&=&\alpha_\mathfrak{h}\circ f.
  \end{eqnarray}
  \end{defi}
}

 \begin{defi}{\rm(\cite[Definition 4.1]{sheng3})}\label{defi:hom-lie representation}
 A {\bf representation} of a Hom-Lie algebra $(\g,[\cdot,\cdot],\alpha)$ on
 a vector space $V$ with respect to $\beta\in\gl(V)$ is a linear map
  $\rho:\g\longrightarrow \gl(V)$, such that for all
  $x,y\in \g$, the following equalities are satisfied:
\begin{eqnarray}
\label{hom-lie-rep-1}\rho(\alpha(x))\circ \beta&=&\beta\circ \rho(x),\\
\label{hom-lie-rep-2}\rho([x,y])\circ \beta&=&\rho(\alpha(x))\circ\rho(y)-\rho(\alpha(y))\circ\rho(x).
\end{eqnarray}
  \end{defi}
We denote a representation by $(V,\beta,\rho)$. For all $x\in\mathfrak{g}$, we define $\ad_{x}:\mathfrak{g}\lon \mathfrak{g}$ by
\begin{eqnarray}
\ad_{x}(y)=[x,y],\quad\forall y \in \mathfrak{g}.
\end{eqnarray}
Then $\ad:\g\longrightarrow\gl(\frak g)$ is a representation of the Hom-Lie algebra $(\mathfrak{g},[\cdot,\cdot],\alpha)$ on $\g$ with respect to $\alpha$, which is called the {\bf adjoint representation}.

\begin{lem}\label{lem:semidirectp}
Let $(\g,[\cdot,\cdot],\alpha)$ be a Hom-Lie algebra, $(V,\beta)$  a vector space with a linear transformation and $\rho:
\g\rightarrow \gl(V)$ a linear
map. Then $(V,\beta,\rho)$ is a representation of $(\g,[\cdot,\cdot],\alpha)$ if and only if $(\g\oplus V,[\cdot,\cdot]_\rho,\alpha+\beta)$ is a Hom-Lie algebra, where $[\cdot,\cdot]_\rho$ and $\alpha+\beta$ are defined by
\begin{eqnarray}\label{eq:sum}
[x+u,y+v]_{\rho}&=&[x,y]+\rho(x)v-\rho(y)u,\\
(\alpha+\beta)(x+u)&=&\alpha(x)+\beta(u).
\end{eqnarray}
for all $x,y\in \g,~u,v\in V$.
\end{lem}

This Hom-Lie algebra is called the semidirect product of $(\g,[\cdot,\cdot],\alpha)$ and $(V,\beta)$, and denoted by $\g\ltimes_\rho V.$

Given a representation $(V,\beta,\rho)$ of a $\Hom$-Lie algebra $(\g,[\cdot,\cdot],\alpha)$, the set of $k$-cochains is given by
\begin{equation}
 \huaC^k(\g;V)=\Hom(\wedge^k \g,V),\quad
 \forall k\geq 0.
\end{equation}
Define
 $\dM:\huaC^k(\g;V)\longrightarrow \huaC^{k+1}(\g;V)$  by
\begin{eqnarray}
  &&(\dM f)(x_1,\cdots,x_{k+1})=\sum_{i=1}^{k+1}(-1)^{i+1}\rho(x_i)\big{(}f(\alpha^{-1}(x_1),\cdots,\widehat{\alpha^{-1}(x_i)},\cdots,\alpha^{-1}(x_{k+1}))\big{)}\nonumber\\
  &&+\sum_{1\le i<j\le k+1}(-1)^{i+j}\beta f([\alpha^{-2}(x_i),\alpha^{-2}(x_j)],\alpha^{-1}(x_1),\cdots,\widehat{\alpha^{-1}(x_i)},\cdots,\widehat{\alpha^{-1}(x_j)},\cdots,\alpha^{-1}(x_{k+1})).\nonumber
\end{eqnarray}
Then we have $$\dM\circ\dM=0.$$ See \cite{caisheng} for more details.
Denote by $\huaZ^k(\g;V)$ the set of closed $k$-cochains, $\huaB^k(\g;V)$ the set of exact $k$-cochains, $\huaH^k(\g;V)=\huaZ^k(\g;V)/\huaB^k(\g;V)$ the corresponding cohomology groups.

Let $(V,\beta,\rho)$ be a representation of a Hom-Lie algebra $(\g,[\cdot,\cdot],\alpha)$. In the sequel, we always assume that $\beta$ is invertible. Define $\rho^*:\g\longrightarrow\gl(V^*)$ as usual by
$$\langle \rho^*(x)(\xi),u\rangle=-\langle\xi,\rho(x)(u)\rangle,\quad\forall x\in\g,u\in V,\xi\in V^*.$$
However, in general $\rho^*$ is not a representation of $\g$ anymore. Define $$\rho^\star:\g\longrightarrow\gl(V^*)$$ by
\begin{equation}\label{eq:new1}
 \rho^\star(x)(\xi):=\rho^*(\alpha(x))\big{(}(\beta^{-2})^*(\xi)\big{)},\quad\forall x\in\g,\xi\in V^*.
\end{equation}

\emptycomment{
More precisely, we have
\begin{eqnarray}\label{eq:new1gen}
\langle\rho^\star(x)(\xi),u\rangle=-\langle\xi,\rho(\alpha^{-1}(x))(\beta^{-2}(u))\rangle,\quad\forall x\in\g, u\in V, \xi\in V^*.
\end{eqnarray}
}

\begin{lem}{\rm(\cite[Lemma 3.4]{Cai-Sheng})}\label{lem:dualrep}
 Let $(V,\beta,\rho)$ be a representation of a Hom-Lie algebra $(\g,[\cdot,\cdot],\alpha)$. Then $(V^*,(\beta^{-1})^*,\rho^\star)$ is a representation of $(\g,[\cdot,\cdot],\alpha)$, where $\rho^\star:\g\longrightarrow\gl(V^*)$ is defined by \eqref{eq:new1}.
\end{lem}

The representation is called the {\bf dual representation} of the representation $(V,\beta,\rho)$.

\section{Representations of Hom-pre-Lie algebras}
Let $(A,\cdot,\alpha)$ be a Hom-pre-Lie algebra. We always assume that it is regular, i.e. $\alpha$ is invertible. The commutator $[x,y]_C=x\cdot y-y\cdot x$ gives a Hom-Lie algebra $(A,[\cdot,\cdot]_C,\alpha)$, which is denoted by $A^C$ and called  the {\bf sub-adjacent Hom-Lie algebra} of $(A,\cdot,\alpha)$.
\begin{defi}
 A morphism from a Hom-pre-Lie algebra $(A,\cdot,\alpha)$ to a Hom-pre-Lie algebra $(A',\cdot',\alpha')$ is a linear map $f:A\longrightarrow A'$ such that for all
  $x,y\in A$, the following equalities are satisfied:
\begin{eqnarray}
\label{homo-1}f(x\cdot y)&=&f(x)\cdot' f(y),\hspace{3mm}\forall x,y\in A,\\
\label{homo-2}f\circ \alpha&=&\alpha'\circ f.
\end{eqnarray}
\end{defi}
\begin{defi}\label{defi:hom-pre representation}
 A {\bf representation} of a Hom-pre-Lie algebra $(A,\cdot,\alpha)$ on a vector space $V$ with respect to $\beta\in\gl(V)$ consists of a pair $(\rho,\mu)$, where $\rho:A\longrightarrow \gl(V)$ is a representation of the sub-adjacent Hom-Lie algebra $A^C$ on $V$ with respect to $\beta\in\gl(V)$, and $\mu:A\longrightarrow \gl(V)$ is a linear map, for all $x,y\in A$, satisfying:
\begin{eqnarray}
\label{rep-1}\beta\circ \mu(x)&=&\mu(\alpha(x))\circ \beta,\\
\label{rep-2}\mu(\alpha(y))\circ\mu(x)-\mu(x\cdot y)\circ \beta&=&\mu(\alpha(y))\circ\rho(x)-\rho(\alpha(x))\circ\mu(y).
\end{eqnarray}
\end{defi}
We denote a representation of a Hom-pre-Lie algebra $(A,\cdot,\alpha)$ by $(V,\beta,\rho,\mu)$. We define a bilinear operation $\cdot_\ltimes:\otimes^2(A\oplus V)\lon(A\oplus V)$ by
\begin{equation}
(x+u)\cdot_\ltimes (y+v):=x\cdot y+\rho(x)(v)+\mu(y)(u),\quad \forall x,y \in A,  u,v\in V.
\end{equation}\label{eq:12.8}
and a linear map $\alpha+\beta:A\oplus V\longrightarrow A\oplus V$ by
\begin{equation}
(\alpha+\beta)(x+u):=\alpha(x)+\beta(u),\quad \forall x\in A,  u\in V.
\end{equation}
It is straightforward to obtain the following result.
\begin{pro}\label{direct-product}
 With the above notations, $(A\oplus V,\cdot_\ltimes,\alpha+\beta)$ is a Hom-pre-Lie algebra, which is denoted by $A\ltimes_{(\rho,\mu)}V$ and called  the  semi-direct product of the  Hom-pre-Lie algebra $(A,\cdot,\alpha)$ and the representation $(V,\beta,\rho,\mu)$.
\end{pro}

\begin{ex}\label{ex-hom-pre-lie}
{\rm
Let $\{e_1,e_2\}$ be a basis of $2$-dimensional vector space $A$. Define a product on $A$ by
\begin{eqnarray*}
e_1\cdot e_2=0,\quad e_2\cdot e_1=e_1,\quad e_1\cdot e_1=0,\quad e_2\cdot e_2=e_1+e_2.
\end{eqnarray*}
Let $\alpha=\begin{pmatrix} 1 & 1 \\
                            0 & 1 \end{pmatrix}.$ More precisely,
$\alpha(e_1)=e_1,~\alpha(e_2)=e_1+e_2$.
Then $(A,\cdot,\alpha)$ is a Hom-pre-Lie algebra.
}
\end{ex}

\begin{ex}{\rm
Let $(A,\cdot,\alpha)$ be a Hom-pre-Lie algebra. Define $L,R:A\longrightarrow \gl(A)$ by
$$L_xy=x\cdot y,\quad R_xy=y\cdot x, \quad \forall x,y \in A.$$
Then $(A,\alpha,L,R)$ is a representation of $(A,\cdot,\alpha)$, which is called the {\bf regular representation}.
}
\end{ex}
\begin{ex}{\rm
Let $(A,\cdot,\alpha)$ be a Hom-pre-Lie algebra. Then $(\mathbb{R},1,0,0)$ is a representation of $(A,\cdot,\alpha)$, which is called the {\bf trivial representation}.
}
\end{ex}

\begin{pro}\label{important-rep}
Let $(V,\beta,\rho,\mu)$ be a representation of a Hom-pre-Lie algebra $(A,\cdot,\alpha)$. Then $(V,\beta,\rho-\mu)$ is a representation of the sub-adjacent Hom-Lie algebra $A^C$.
\end{pro}

\begin{proof}
By Proposition \ref{direct-product}, we have the semidirect product Hom-pre-Lie algebra $A\ltimes_{(\rho,\mu)}V$. Consider its sub-adjacent Hom-Lie algebra structure $(A\oplus V,[\cdot,\cdot]_{C},\alpha+\beta)$, we have
\begin{eqnarray*}
[x+u,y+v]_{C}&=&(x+u)\cdot_\ltimes(y+v)-(y+v)\cdot_\ltimes(x+u)\\
                     &=&x\cdot_A y+\rho(x)(v)+\mu(y)(u)-y\cdot_A x-\rho(y)(u)-\mu(x)(v)\\
                     &=&[x,y]_C+(\rho(x)-\mu(x))v-(\rho(y)-\mu(y))u,
\end{eqnarray*}
and
\begin{eqnarray*}
(\alpha+\beta)(x+u)=\alpha(x)+\beta(u).
\end{eqnarray*}
By Lemma \ref{lem:semidirectp}, we deduce that $(V,\beta,\rho-\mu)$ is a representation of the Hom-Lie algebra $A^C$.
\end{proof}

Let $(V,\beta,\rho,\mu)$ be a representation of a Hom-pre-Lie algebra $(A,\cdot,\alpha)$. In the sequel, we always assume that $\beta$ is invertible. For all $x\in A,u\in V,\xi\in V^*$, define $\rho^*:A\longrightarrow\gl(V^*)$ and $\mu^*:A\longrightarrow\gl(V^*)$ as usual by
$$\langle \rho^*(x)(\xi),u\rangle=-\langle\xi,\rho(x)(u)\rangle,\quad \langle \mu^*(x)(\xi),u\rangle=-\langle\xi,\mu(x)(u)\rangle.$$
Then define $\rho^\star:A\longrightarrow\gl(V^*)$ and $\mu^\star:A\longrightarrow\gl(V^*)$ by
\begin{eqnarray}
  \label{eq:1.3}\rho^\star(x)(\xi):=\rho^*(\alpha(x))\big{(}(\beta^{-2})^*(\xi)\big{)},\\
   \label{eq:1.4}\mu^\star(x)(\xi):=\mu^*(\alpha(x))\big{(}(\beta^{-2})^*(\xi)\big{)}.
\end{eqnarray}

\begin{thm}\label{dual-rep}
Let $(V,\beta,\rho,\mu)$ be a representation of a Hom-pre-Lie algebra $(A,\cdot,\alpha)$. Then $(V^*,(\beta^{-1})^*,\rho^\star-\mu^\star,-\mu^\star)$ is a representation of $(A,\cdot,\alpha)$, which is called the {\bf dual representation} of $(V,\beta,\rho,\mu)$.
\end{thm}
\begin{proof}
Since $(V,\beta,\rho,\mu)$ is a representation of a Hom-pre-Lie algebra $(A,\cdot,\alpha)$. By Proposition \ref{important-rep}, we deduce that $(V,\beta,\rho-\mu)$ is a representation of the sub-adjacent Hom-Lie algebra $A^C$. Moreover, by Lemma \ref{lem:dualrep}, we obtain that $$(V^*,(\beta^{-1})^*,(\rho-\mu)^\star=\rho^\star-\mu^\star)$$ is a representation of the Hom-Lie algebra $A^C$.

By \eqref{rep-1} and \eqref{eq:1.4}, for all $x\in A,\xi\in V^*$, we have
\begin{eqnarray*}
-\mu^\star(\alpha(x))\big((\beta^{-1})^*(\xi)\big)&=&-\mu^*(\alpha^{2}(x))\big((\beta^{-3})^*(\xi)\big)\\
&=&(\beta^{-1})^*\big(-\mu^*(\alpha(x))\big)\big((\beta^{-2})^*(\xi)\big)\\
&=&(\beta^{-1})^*\big(-\mu^\star(x)(\xi)\big),
\end{eqnarray*}
which implies that
\begin{equation}\label{dual-rep-1}
-\mu^\star{(}\alpha(x){)}\circ(\beta^{-1})^*=(\beta^{-1})^*\circ(-\mu^\star(x)).
\end{equation}

By \eqref{hom-lie-rep-1}, \eqref{rep-1}, \eqref{rep-2} and \eqref{eq:1.4}, for all $x,y\in A,\xi\in V^*$ and $u\in V$, we have
\begin{eqnarray*}
&&\Big\langle\mu^\star(x\cdot y)\big((\beta^{-1})^*(\xi)\big),u\Big\rangle\\
&=&\Big\langle\mu^*\big(\alpha(x)\cdot \alpha(y)\big)\big((\beta^{-3})^*(\xi)\big),u\Big\rangle\\
&=&-\Big\langle(\beta^{-3})^*(\xi),\mu\big(\alpha(x)\cdot \alpha(y)\big)(u)\Big\rangle\\
&=&-\Big\langle(\beta^{-3})^*(\xi),\mu(\alpha^{2}(y))\big(\mu(\alpha (x))(\beta^{-1}(u))\big)-\mu(\alpha^{2}(y))\big(\rho(\alpha(x))(\beta^{-1}(u))\big)\\
&&+\rho(\alpha^{2}(x))\big(\mu(\alpha(y))(\beta^{-1}(u))\big)\Big\rangle\\
&=&-\Big\langle(\beta^{-4})^*(\xi),\mu(\alpha^{3}(y))\big(\mu(\alpha^{2}(x))(u)\big)-\mu(\alpha^{3}(y))\big(\rho(\alpha^{2}(x))(u)\big)+\rho(\alpha^{3}(x))\big(\mu(\alpha^{2}(y))(u)\big)\Big\rangle\\
&=&-\Big\langle\mu^*(\alpha^{2}(x))\big(\mu^*(\alpha^{3}(y))((\beta^{-4})^*(\xi))\big)-\rho^*(\alpha^{2}(x))\big(\mu^*(\alpha^{3}(y))((\beta^{-4})^*(\xi))\big)\\
&&+\mu^*(\alpha^{2}(y))\big(\rho^*(\alpha^{3}(x))((\beta^{-4})^*(\xi))\big),u\Big\rangle\\
&=&-\Big\langle\mu^*(\alpha^{2}(x))\big(\mu^\star(\alpha^{2}(y))((\beta^{-2})^*(\xi))\big)-\rho^*(\alpha^{2}(x))\big(\mu^\star(\alpha^{2}(y))((\beta^{-2})^*(\xi))\big)\\
&&+\mu^*(\alpha^{2}(y))\big(\rho^\star(\alpha^{2}(x))((\beta^{-2})^*(\xi))\big),u\Big\rangle\\
&=&-\Big\langle\mu^*(\alpha^{2}(x))\big((\beta^{-2})^*(\mu^\star(y)(\xi))\big)-\rho^*(\alpha^{2}(x))\big((\beta^{-2})^*(\mu^\star(y)(\xi))\big)\\
&&+\mu^*(\alpha^{2}(y))\big((\beta^{-2})^*(\rho^\star(x)(\xi))\big),u\Big\rangle\\
&=&-\Big\langle\mu^\star(\alpha (x))\big{(}\mu^\star(y)(\xi)\big{)}-\rho^\star(\alpha (x))\big{(}\mu^\star(y)(\xi)\big{)}+\mu^\star(\alpha (y))\big{(}\rho^\star(x)(\xi)\big{)},u\Big\rangle,
\end{eqnarray*}
which implies that
\begin{equation}\label{dual-rep-2}
\mu^\star(\alpha(y))\circ\mu^\star(x)+\mu^\star(x\cdot y)\circ(\beta^{-1})^*=-\mu^\star(\alpha(y))\circ(\rho^\star-\mu^\star)(x)+(\rho^\star-\mu^\star)(\alpha(x))\circ\mu^\star(y).
\end{equation}
By \eqref{dual-rep-1} and \eqref{dual-rep-2}, we deduce that $(V^*,(\beta^{-1})^*,\rho^\star-\mu^\star,-\mu^\star)$ is a representation of $(A,\cdot,\alpha)$.
\end{proof}

Consider the dual representation of the regular representation, we have
\begin{cor}
 Let  $(A,\cdot,\alpha)$ be a Hom-pre-Lie algebra. Then $(A^*,(\alpha^{-1})^*,\ad^\star=L^\star-R^\star,-R^\star)$ is a representation of $(A,\cdot,\alpha)$.
\end{cor}

\begin{pro}\label{dual-2-identity}
Let $(V,\beta,\rho,\mu)$ be a representation of a Hom-pre-Lie algebra $(A,\cdot,\alpha)$. Then the dual representation of $(V^*,(\beta^{-1})^*,\rho^\star,\mu^\star)$ is $(V,\beta,\rho,\mu).$
\end{pro}

\begin{proof}
It is obviously that $(V^\ast)^\ast=V$, $(((\beta^{-1})^\ast)^{-1})^\ast=\beta$. By Lemma $3.7$ in \cite{Cai-Sheng}, we obtain $(\rho^\star)^\star=\rho$. Similarly, we have $(\mu^\star)^\star=\mu$. The proof is finished.
\end{proof}

\begin{pro}
Let $(V,\beta,\rho,\mu)$ be a representation of a Hom-pre-Lie algebra $(A,\cdot,\alpha)$. Then the following conditions are equivalent:
\begin{itemize}
\item [$\rm(i)$]  $(V,\beta,\rho-\mu,-\mu)$ is a representation of the Hom-pre-Lie algebra $(A,\cdot,\alpha)$,
\item[$\rm(ii)$] $(V^\ast,(\beta^{-1})^\ast,\rho^\star,\mu^\star)$ is a representation of the Hom-pre-Lie algebra $(A,\cdot,\alpha)$,
\item[$\rm(iii)$] $\mu(\alpha(x))\circ \mu(y)=-\mu(\alpha(y))\circ \mu(x)$, for all $x,y\in A$.
\end{itemize}
\end{pro}
\begin{proof}
By Theorem \ref{dual-rep} and Proposition \ref{dual-2-identity}, we obtain that condition $\rm(i)$ is equivalent to condition $\rm(ii)$. Since $(V,\beta,\rho,\mu)$ is a representation of a Hom-pre-Lie algebra $(A,\cdot,\alpha)$, we deduce that condition $\rm(i)$ is equivalent to condition $\rm(iii)$.
\end{proof}

At the end of this section, we show that the tensor product of two representations of a Hom-pre-Lie algebra is still a representation.

\begin{pro}
Let $(A,\cdot,\alpha)$ be a Hom-pre-Lie algebra, $(V,\beta_V,\rho_V,\mu_V)$ and $(W,\beta_W,\rho_W,\mu_W)$ its representations. Then $(V\otimes W,\beta_V\otimes \beta_W,\rho_V\otimes \beta_W+\beta_V\otimes (\rho_W-\mu_W),\mu_V\otimes \beta_W)$ is a representation of $(A,\cdot,\alpha)$.
\end{pro}
\begin{proof}
Since $(V,\beta_V,\rho_V,\mu_V)$ and $(W,\beta_W,\rho_W,\mu_W)$ are representations of a Hom-pre-Lie algebra $(A,\cdot,\alpha)$, for all $x\in A$, we have
\begin{eqnarray*}&&\Big(\rho_V\otimes \beta_W+\beta_V\otimes (\rho_W-\mu_W)\Big)(\alpha(x))\circ (\beta_V\otimes \beta_W)\\
&&-(\beta_V\otimes \beta_W)\circ \Big(\rho_V\otimes \beta_W+\beta_V\otimes (\rho_W-\mu_W)\Big)(x)\\
&=&\big(\rho_V(\alpha(x))\circ \beta_V \big)\otimes(\beta_W\circ \beta_W)+(\beta_V\circ \beta_V)\otimes \big(\rho_W(\alpha(x))\circ \beta_W\big)\\
&&-(\beta_V\circ \beta_V)\otimes \big(\mu_W(\alpha(x))\circ \beta_W\big)-\big(\beta_V\circ \rho_V(x)\big)\otimes(\beta_W\circ \beta_W)\\
&&-(\beta_V\circ \beta_V)\otimes \big(\beta_W \circ \rho_W(x)\big)+(\beta_V\circ \beta_V)\otimes \big(\beta_W\circ \mu_W(x)\big)\\
&=&0,
\end{eqnarray*}
which implies that
\begin{equation}\label{tensor-rep-1}
\Big(\rho_V\otimes \beta_W+\beta_V\otimes (\rho_W-\mu_W)\Big)(\alpha(x))\circ (\beta_V\otimes \beta_W)=(\beta_V\otimes \beta_W)\circ \Big(\rho_V\otimes \beta_W+\beta_V\otimes (\rho_W-\mu_W)\Big)(x).
\end{equation}
Similarly, we have
\begin{equation}\label{tensor-rep-2}
(\mu_V\otimes \beta_W)(\alpha(x))\circ (\beta_V\otimes \beta_W)=(\beta_V\otimes \beta_W)\circ (\mu_V\otimes \beta_W )(x).
\end{equation}
For all $x,y\in A$, by straightforward computations, we have
\begin{eqnarray*}&&\Big(\rho_V\otimes \beta_W+\beta_V\otimes (\rho_W-\mu_W)\Big)([x,y])\circ (\beta_V\otimes \beta_W)\\
&&-\Big(\rho_V\otimes \beta_W+\beta_V\otimes (\rho_W-\mu_W)\Big)(\alpha(x))\circ \Big(\rho_V\otimes \beta_W+\beta_V\otimes (\rho_W-\mu_W)\Big)(y)\\
&&+\Big(\rho_V\otimes \beta_W+\beta_V\otimes (\rho_W-\mu_W)\Big)(\alpha(y))\circ \Big(\rho_V\otimes \beta_W+\beta_V\otimes (\rho_W-\mu_W)\Big)(x)\\
&=&\big(\rho_V([x,y])\circ \beta_V \big)\otimes(\beta_W\circ \beta_W)+(\beta_V\circ \beta_V)\otimes \big(\rho_W([x,y])\circ \beta_W\big)\\
&&-(\beta_V\circ \beta_V)\otimes \big(\mu_W([x,y])\circ \beta_W\big)-\big(\rho_V(\alpha(x))\circ \rho_V(y)\big)\otimes (\beta_W\circ \beta_W)\\
&&-\big(\rho_V(\alpha(x))\circ \beta_V\big)\otimes \big(\beta_W\circ \rho_W(y)\big)+\big(\rho_V(\alpha(x))\circ \beta_V \big)\otimes \big(\beta_W\circ \mu_W(y)\big)\\
&&-\big(\beta_V\circ \rho_V(y)\big)\otimes \big(\rho_W(\alpha(x))\circ \beta_W\big)-(\beta_V\circ \beta_V)\otimes \big(\rho_W(\alpha(x))\circ \rho_W(y)\big)\\
&&+(\beta_V\circ \beta_V)\otimes \big(\rho_W(\alpha(x))\circ \mu_W(y)\big)+\big(\beta_V\circ \rho_V(y)\big)\otimes \big(\mu_W(\alpha(x))\circ \beta_W\big)\\
&&+(\beta_V\circ \beta_V)\otimes \big(\mu_W(\alpha(x))\circ \rho_W(y)\big)-(\beta_V\circ \beta_V)\otimes \big(\mu_W(\alpha(x))\circ \mu_W(y)\big)\\
&&+\big(\rho_V(\alpha(y))\circ \rho_V(x)\big)\otimes (\beta_W\circ \beta_W)+\big(\rho_V(\alpha(y))\circ \beta_V\big)\otimes \big(\beta_W\circ \rho_W(x)\big)\\
&&-\big(\rho_V(\alpha(y))\circ \beta_V \big)\otimes \big(\beta_W\circ \mu_W(x)\big)+\big(\beta_V\circ \rho_V(x)\big)\otimes \big(\rho_W(\alpha(y))\circ \beta_W\big)\\
&&+(\beta_V\circ \beta_V)\otimes \big(\rho_W(\alpha(y))\circ \rho_W(x)\big)-(\beta_V\circ \beta_V)\otimes \big(\rho_W(\alpha(y))\circ \mu_W(x)\big)\\
&&-\big(\beta_V\circ \rho_V(x)\big)\otimes \big(\mu_W(\alpha(y))\circ \beta_W\big)-(\beta_V\circ \beta_V)\otimes \big(\mu_W(\alpha(y))\circ \rho_W(x)\big)\\
&&+(\beta_V\circ \beta_V)\otimes \big(\mu_W(\alpha(y))\circ \mu_W(x)\big)\\
&=&0,
\end{eqnarray*}
which implies that
\begin{eqnarray}\label{tensor-rep-3}
\nonumber&&\Big(\rho_V\otimes \beta_W+\beta_V\otimes (\rho_W-\mu_W)\Big)([x,y])\circ (\beta_V\otimes \beta_W)\\
\nonumber&=&\Big(\rho_V\otimes \beta_W+\beta_V\otimes (\rho_W-\mu_W)\Big)(\alpha(x))\circ \Big(\rho_V\otimes \beta_W+\beta_V\otimes (\rho_W-\mu_W)\Big)(y)\\
&&-\Big(\rho_V\otimes \beta_W+\beta_V\otimes (\rho_W-\mu_W)\Big)(\alpha(y))\circ \Big(\rho_V\otimes \beta_W+\beta_V\otimes (\rho_W-\mu_W)\Big)(x).
\end{eqnarray}
Similarly, we have
\begin{eqnarray}\label{tensor-rep-4}
\nonumber &&(\mu_V\otimes \beta_W)(\alpha(y))\circ \Big(\rho_V\otimes \beta_W+\beta_V\otimes (\rho_W-\mu_W)\Big)(x)\\
\nonumber &&-\Big(\rho_V\otimes \beta_W+\beta_V\otimes (\rho_W-\mu_W)\Big)(\alpha(x))\circ (\mu_V\otimes \beta_W)(y)\\
&=&(\mu_V\otimes \beta_W)(\alpha(y))\circ (\mu_V\otimes \beta_W)(x)-(\mu_V\otimes \beta_W)(x\cdot y)\circ (\beta_V\otimes \beta_W).
\end{eqnarray}
By \eqref{tensor-rep-1}, \eqref{tensor-rep-2}, \eqref{tensor-rep-3}, \eqref{tensor-rep-4}, we deduce that $(V\otimes W,\beta_V\otimes \beta_W,\rho_V\otimes \beta_W+\beta_V\otimes (\rho_W-\mu_W),\mu_V\otimes \beta_W)$ is a representation of $(A,\cdot,\alpha)$.
\end{proof}

\section{Cohomologies of Hom-pre-Lie algebras}

In this section, we establish the cohomology theory of Hom-pre-Lie algebras. See \cite{DA} for more details for the cohomology theory of pre-Lie algebras.

Let $(V,\beta,\rho,\mu)$ be a representation of a Hom-pre-Lie algebra $(A,\cdot,\alpha)$. We define $\Ad^{\beta}_{\alpha}\in \gl(\Hom(A,V))$ by
\begin{equation}\label{eq:2.1}
\Ad^{\beta}_{\alpha}(f)=\beta\circ f\circ \alpha^{-1},\quad \forall f\in \Hom(A,V),
\end{equation}
and define $\hat{\rho}:A^C\longrightarrow \gl(\Hom(A,V))$ by
\begin{equation}\label{eq:2.2}
\hat{\rho}(x)(f)(y)=\rho(x)f(\alpha^{-1}(y))+\mu(y)f(\alpha^{-1}(x))-\Ad^{\beta}_{\alpha}(f)(\alpha^{-1}(x\cdot y)),
\end{equation}
for all $f\in \Hom(A,V), x,y\in A$.
\begin{thm}
Let $(A,\cdot,\alpha)$ be a Hom-pre-Lie algebra and $(V,\beta,\rho,\mu)$ its representation. Then $\hat{\rho}$ is a representation of the sub-adjacent Hom-Lie algebra $A^C$ on the vector space $\Hom(A,V)$ with respect to $\Ad^{\beta}_{\alpha}$.
\end{thm}

\begin{proof}
By \eqref{hom-lie-rep-1} and \eqref{rep-1}, for all $x,y\in A, f\in \Hom(A,V)$, we have
\begin{eqnarray*}&&\Big(\hat{\rho}(\alpha(x))\circ \Ad^{\beta}_{\alpha}-\Ad^{\beta}_{\alpha}\circ \hat{\rho}(x)\Big)(f)(y)\\
&=&\rho(\alpha(x))\Ad^{\beta}_{\alpha}(f)(\alpha^{-1}(y))+\mu(y)\Ad^{\beta}_{\alpha}(f)(x)-(\Ad^{\beta}_{\alpha})^2(f)\big(\alpha^{-1}(\alpha(x)\cdot y)\big)\\
&&-\beta\Big(\rho(x)f(\alpha^{-2}(y))+\mu(\alpha^{-1}(y))f(\alpha^{-1}(x))-\Ad^{\beta}_{\alpha}(f)(\alpha^{-1}(x\cdot \alpha^{-1}(y)))\Big)\\
&=&\rho(\alpha(x))\beta f(\alpha^{-2}(y))+\mu(y)\beta f(\alpha^{-1}(x))-\beta^2 f\big(\alpha^{-2}(x\cdot \alpha^{-1}(y))\big)\\
&&-\beta\rho(x)f(\alpha^{-2}(y))-\beta\mu(\alpha^{-1}(y))f(\alpha^{-1}(x))+\beta^2f\big(\alpha^{-2}(x\cdot_A \alpha^{-1}(y))\big)\\
&=&0.
\end{eqnarray*}
Thus, we have
\begin{equation}\label{eq:1}
\Ad^{\beta}_{\alpha}\circ \hat{\rho}(x)=\hat{\rho}(\alpha(x))\circ \Ad^{\beta}_{\alpha}.
\end{equation}
For all $x,y,z\in A, f\in \Hom(A,V)$, by Definition \ref{defi:hom-pre representation} and the definition of a Hom-pre-Lie algebra, we have
\begin{eqnarray*}&&\Big(\hat{\rho}(\alpha(x))\circ \hat{\rho}(y)-\hat{\rho}(\alpha(y))\circ \hat{\rho}(x)-\hat{\rho}([x,y]_C)\circ\Ad^{\beta}_{\alpha}\Big)(f)(z)\\
&=&\rho(\alpha(x))\big(\hat{\rho}(y)(f)\big)(\alpha^{-1}(z))+\mu(z)\big(\hat{\rho}(y)(f)\big)(x)-\Ad^{\beta}_{\alpha}\big(\hat{\rho}(y)(f)\big)(x\cdot \alpha^{-1}(z))\\
&&-\rho(\alpha(y))\big(\hat{\rho}(x)(f)\big)(\alpha^{-1}(z))-\mu(z)\big(\hat{\rho}(x)(f)\big)(y)+\Ad^{\beta}_{\alpha}\big(\hat{\rho}(x)(f)\big)(y\cdot \alpha^{-1}(z))\\
&&-\rho([x,y]_C)\Ad^{\beta}_{\alpha}(f)(\alpha^{-1}(z))-\mu(z)\Ad^{\beta}_{\alpha}(f)(\alpha^{-1}([x,y]_C))+(\Ad^{\beta}_{\alpha})^2(f)(\alpha^{-1}([x,y]_C\cdot z))\\
&=&\rho(\alpha(x))\rho(y)f(\alpha^{-2}(z))+\rho(\alpha(x))\mu(\alpha^{-1}(z))f(\alpha^{-1}(y))-\rho(\alpha(x))\beta f(\alpha^{-2}(y)\cdot \alpha^{-3}(z))\\
&&+\mu(z)\rho(y)f(\alpha^{-1}(x))+\mu(z)\mu(x)f(\alpha^{-1}(y))-\mu(z)\beta f(\alpha^{-2}(y\cdot x))\\
&&-\beta\rho(y)f(\alpha^{-2}(x)\cdot \alpha^{-3}(z))-\beta\mu(\alpha^{-1}(x)\cdot \alpha^{-2}(z))f(\alpha^{-1}(y))\\
&&+\beta^2 f\big(\alpha^{-2}(y)\cdot (\alpha^{-3}(x)\cdot \alpha^{-4}(z))\big)
-\rho(\alpha(y))\rho(x)f(\alpha^{-2}(z))\\
&&-\rho(\alpha(y))\mu(\alpha^{-1}(z))f(\alpha^{-1}(x))
+\rho(\alpha(y))\beta f(\alpha^{-2}(x)\cdot \alpha^{-3}(z))
-\mu(z)\rho(x)f(\alpha^{-1}(y))\\
&&-\mu(z)\mu(y)f(\alpha^{-1}(x))+\mu(z)\beta f(\alpha^{-2}(x\cdot y))
+\beta\rho(x)f(\alpha^{-2}(y)\cdot \alpha^{-3}(z))\\
&&+\beta\mu(\alpha^{-1}(y)\cdot \alpha^{-2}(z))f(\alpha^{-1}(x))-\beta^2 f\big(\alpha^{-2}(x)\cdot (\alpha^{-3}(y)\cdot \alpha^{-4}(z))\big)\\
&&-\rho([x,y]_C)\beta f(\alpha^{-2}(z))-\mu(z)\beta f(\alpha^{-2}([x,y]_C))+\beta^2 f(\alpha^{-3}([x,y]_C\cdot z))\\
&=&\rho(\alpha(x))\rho(y)f(\alpha^{-2}(z))-\rho(\alpha(y))\rho(x)f(\alpha^{-2}(z))\\
&&-\mu(z)\beta f(\alpha^{-2}(y\cdot x))+\mu(z)\beta f(\alpha^{-2}(x\cdot y))\\
&&+\beta^2f\big(\alpha^{-2}(y)\cdot (\alpha^{-3}(x)\cdot \alpha^{-4}(z))\big)-\beta^2f\big(\alpha^{-2}(x)\cdot (\alpha^{-3}(y)\cdot \alpha^{-4}(z))\big)\\
&&-\rho([x,y]_C)\beta f(\alpha^{-2}(z))-\mu(z)\beta f(\alpha^{-2}([x,y]_C))+\beta^2 f(\alpha^{-3}([x,y]_C\cdot z))\\
&=&0,
\end{eqnarray*}
which implies that
\begin{equation}\label{eq:2}
\hat{\rho}([x,y]_C)\circ\Ad^{\beta}_{\alpha}=\hat{\rho}(\alpha(x))\circ\hat{\rho}(y)-\hat{\rho}(\alpha(y))\circ\hat{\rho}(x).
\end{equation}
By \eqref{eq:1} and \eqref{eq:2}, we deduce that $\hat{\rho}$ is a representation of the Hom-Lie algebra $A^C$ on $\Hom(A,V)$ with respect to $\Ad^{\beta}_{\alpha}$.
\end{proof}

Let $(V,\beta,\rho,\mu)$ be a representation of a Hom-pre-Lie algebra $(A,\cdot,\alpha)$. The set of $(n+1)$-cochains is given by
\begin{equation}
 C^{n+1}(A;V)=\Hom(\wedge^n A\otimes A,V),\quad
 \forall n\geq 0.
\end{equation}
For all $f \in C^n(A;V),~ x_1,\dots,x_{n+1} \in A$, define the operator
$\partial:C^n(A;V)\longrightarrow C^{n+1}(A;V)$ by
\begin{eqnarray}\label{eq:12.1}
 &&(\partial f)(x_1,\dots,x_{n+1})\\
 \nonumber&=&\sum_{i=1}^n(-1)^{i+1}\rho(x_i)f(\alpha^{-1}(x_1),\dots,\widehat{\alpha^{-1}(x_i)},\dots,\alpha^{-1}(x_{n+1}))\\
\nonumber &&+\sum_{i=1}^n(-1)^{i+1}\mu(x_{n+1})f(\alpha^{-1}(x_1),\dots,\widehat{\alpha^{-1}(x_i)},\dots,\alpha^{-1}(x_n),\alpha^{-1}(x_i))\\
\nonumber &&-\sum_{i=1}^n(-1)^{i+1}\beta f(\alpha^{-1}(x_1),\dots,\widehat{\alpha^{-1}(x_i)}\dots,\alpha^{-1}(x_n),\alpha^{-2}(x_i)\cdot \alpha^{-2}(x_{n+1}))\\
\nonumber&&+\sum_{1\leq i<j\leq n}(-1)^{i+j}\beta f([\alpha^{-2}(x_i),\alpha^{-2}(x_j)]_C,\alpha^{-1}(x_1),\dots,\widehat{\alpha^{-1}(x_i)},\dots,\widehat{\alpha^{-1}(x_j)},\dots,\alpha^{-1}(x_{n+1})).
\end{eqnarray}
To prove that the operator $\partial $ is indeed a coboundary operator, we need some preparations. For vector spaces $A$ and $V$, we have the natural isomorphism $$\Phi:\Hom(\wedge^{n-1} A^C, \Hom(A,V))\longrightarrow \Hom(\wedge^{n-1} A\otimes A,V)$$ defined by
\begin{equation}\label{eq:3.1}
\Phi(\omega)(x_1,\dots,x_{n-1},x_n)=\omega(x_1,\dots,x_{n-1})(x_n),\quad \forall x_1,\dots,x_{n-1},x_n\in A.
\end{equation}
\begin{pro}\label{pro:commutative}
Let $(V,\beta,\rho,\mu)$ be a representation of a Hom-pre-Lie algebra $(A,\cdot,\alpha)$. Then we have the following commutative diagram
$$\xymatrix{
  \Hom(\wedge^{n-1} A^C, \Hom(A,V))\ar[d]_{\dM}  \ar[r]^{\quad
\Phi} & \Hom(\wedge^{n-1} A\otimes A,V) \ar[d]^{\partial}  \\
     \Hom(\wedge^n A^C, \Hom(A,V))\ar[r]^{\quad
\Phi}   &\Hom(\wedge^n A\otimes A,V),              }$$
where $\dM$ is the coboundary operator of the sub-adjacent Hom-Lie algebra $A^C$ with the coefficient in the representation $(\Hom(A,V),\Ad^{\beta}_{\alpha},\hat{\rho})$ and $\partial$ is defined by \eqref{eq:12.1}.
\end{pro}
\begin{proof}
For all  $\omega\in \Hom(\wedge^{n-1} A^C, \Hom(A,V))$, we have
\begin{eqnarray*}&&\partial (\Phi(\omega))(x_1,\dots,x_{n+1})\\
&=&\sum_{i=1}^n(-1)^{i+1}\rho(x_i)\Phi(\omega)(\alpha^{-1}(x_1),\dots,\widehat{\alpha^{-1}(x_i)},\dots,\alpha^{-1}(x_{n+1}))\\
&&+\sum_{i=1}^n(-1)^{i+1}\mu(x_{n+1})\Phi(\omega)(\alpha^{-1}(x_1),\dots,\widehat{\alpha^{-1}(x_i)},\dots,\alpha^{-1}(x_n),\alpha^{-1}(x_i))\\
&&-\sum_{i=1}^n(-1)^{i+1}\beta \Phi(\omega)(\alpha^{-1}(x_1),\dots,\widehat{\alpha^{-1}(x_i)}\dots,\alpha^{-1}(x_n),\alpha^{-2}(x_i)\cdot \alpha^{-2}(x_{n+1}))\\
&&+\sum_{1\leq i<j\leq n}(-1)^{i+j}\beta \Phi(\omega)([\alpha^{-2}(x_i),\alpha^{-2}(x_j)]_C,\alpha^{-1}(x_1),\dots,\widehat{\alpha^{-1}(x_i)},\dots,\widehat{\alpha^{-1}(x_j)},\dots,\alpha^{-1}(x_{n+1}))\\
&\stackrel{\eqref{eq:3.1}}{=}&\sum_{i=1}^n(-1)^{i+1}\rho(x_i)\omega(\alpha^{-1}(x_1),\dots,\widehat{\alpha^{-1}(x_i)},\dots,\alpha^{-1}(x_n))(\alpha^{-1}(x_{n+1}))\\
&&+\sum_{i=1}^n(-1)^{i+1}\mu(x_{n+1})\omega(\alpha^{-1}(x_1),\dots,\widehat{\alpha^{-1}(x_i)},\dots,\alpha^{-1}(x_n))(\alpha^{-1}(x_i))\\
&&-\sum_{i=1}^n(-1)^{i+1}\beta \omega(\alpha^{-1}(x_1),\dots,\widehat{\alpha^{-1}(x_i)}\dots,\alpha^{-1}(x_n))(\alpha^{-2}(x_i)\cdot \alpha^{-2}(x_{n+1}))\\
&&+\sum_{1\leq i<j\leq n}(-1)^{i+j}\beta \omega([\alpha^{-2}(x_i),\alpha^{-2}(x_j)]_C,\alpha^{-1}(x_1),\dots,\widehat{\alpha^{-1}(x_i)},\dots,\widehat{\alpha^{-1}(x_j)},\dots,\alpha^{-1}(x_n))(\alpha^{-1}(x_{n+1}))\\
&\stackrel{\eqref{eq:2.1}}{=}&\sum_{i=1}^n(-1)^{i+1}\rho(x_i)\omega(\alpha^{-1}(x_1),\dots,\widehat{\alpha^{-1}(x_i)},\dots,\alpha^{-1}(x_n))(\alpha^{-1}(x_{n+1}))\\
&&+\sum_{i=1}^n(-1)^{i+1}\mu(x_{n+1})\omega(\alpha^{-1}(x_1),\dots,\widehat{\alpha^{-1}(x_i)},\dots,\alpha^{-1}(x_n))(\alpha^{-1}(x_i))\\
&&-\sum_{i=1}^n(-1)^{i+1}\Ad^{\beta}_{\alpha}\omega(\alpha^{-1}(x_1),\dots,\widehat{\alpha^{-1}(x_i)}\dots,\alpha^{-1}(x_n))(\alpha^{-1}(x_i)\cdot \alpha^{-1}(x_{n+1}))\\
&&+\sum_{1\leq i<j\leq n}(-1)^{i+j}\Ad^{\beta}_{\alpha} \omega([\alpha^{-2}(x_i),\alpha^{-2}(x_j)]_C,\alpha^{-1}(x_1),\dots,\widehat{\alpha^{-1}(x_i)},\dots,\widehat{\alpha^{-1}(x_j)},\dots,\alpha^{-1}(x_n))(x_{n+1})\\
&\stackrel{\eqref{eq:2.2}}{=}&\sum_{i=1}^n(-1)^{i+1}\hat{\rho}(x_i)\omega(\alpha^{-1}(x_1),\dots,\widehat{\alpha^{-1}(x_i)},\dots,\alpha^{-1}(x_n))(x_{n+1})\\
&&+\sum_{1\leq i<j\leq n}(-1)^{i+j}\Ad^{\beta}_{\alpha} \omega([\alpha^{-2}(x_i),\alpha^{-2}(x_j)]_C,\alpha^{-1}(x_1),\dots,\widehat{\alpha^{-1}(x_i)},\dots,\widehat{\alpha^{-1}(x_j)},\dots,\alpha^{-1}(x_n))(x_{n+1})\\
&=&\Phi(\dM \omega)(x_1,\dots,x_n,x_{n+1}),
\end{eqnarray*}
we deduce $\partial \circ \Phi=\Phi\circ \dM $. The proof is finished.
\end{proof}

\begin{thm}\label{thm:operator}
The operator $\partial:C^n(A;V)\longrightarrow C^{n+1}(A;V)$ defined as above satisfies $\partial\circ\partial=0$.
\end{thm}
\begin{proof}
By Proposition \ref{pro:commutative}, we have $\partial=\Phi\circ \dM \circ \Phi^{-1}$.
Thus, we obtain $\partial^2=\Phi\circ \dM\circ\dM\circ \Phi^{-1}=0$. The proof is finished.
\end{proof}

Denote the set of closed $n$-cochains by $Z^n(A;V)$ and the set of exact $n$-cochains by $B^n(A;V)$. We denote by $H^n(A;V)=Z^n(A;V)/B^n(A;V)$ the corresponding cohomology groups of the Hom-pre-Lie algebra $(A,\cdot,\alpha)$ with the coefficient in the representation $(V,\beta,\rho,\mu)$. There is a close relation between the cohomologies of Hom-pre-Lie algebras and the corresponding sub-adjacent Hom-Lie algebras.
\begin{pro}
Let $(A,\cdot,\alpha)$ be a Hom-pre-Lie algebra and $(V,\beta,\rho,\mu)$ its representation. Then the cohomology groups $\huaH^{n-1}(A^C;\Hom(A,V))$ and $H^n(A;V)$ are isomorphic.
\end{pro}
\proof
By Proposition \ref{pro:commutative}, we deduce that $\Phi$ is an isomorphism from the cochain complex $\big(\huaC^{\ast-1}(A^C;\Hom(A,V)),\dM\big)$ to the cochain complex $\big(C^\ast(A;V),\partial\big)$. Thus, $\Phi$ induces an isomorphism $\Phi_\ast$ from   $\huaH^{\ast-1}(A^C;\Hom(A,V))$ to $H^\ast(A;V).$
\qed\vspace{3mm}

We use $\partial_{\reg}$ and $\partial_T$ to denote the coboundary operator of $(A,\cdot,\alpha)$ with the coefficients in the regular representation and the trivial representation respectively. The corresponding cohomology groups will be denoted by $H^n(A;A)$ and $H^n(A)$ respectively. The two cohomology groups $H^2(A;A)$ and $H^2(A)$ will play important roles in the later study.

\section{Linear deformations of Hom-pre-Lie algebras}
In this section, we study linear deformations of Hom-pre-Lie algebras using the cohomology defined in the previous section, and introduce the notion of a Nijenhuis operator on a Hom-pre-Lie algebra. We show that a Nijenhuis operator gives rise to a trivial deformation.

\vspace{2mm}
Let $(A,\cdot,\alpha)$ be a Hom-pre-Lie algebra. In the sequel, we will also denote the Hom-pre-Lie multiplication $\cdot$ by $\Omega$.

\begin{defi}
Let $(A,\Omega,\alpha)$ be a Hom-pre-Lie algebra and $\omega\in C^2(A;A)$ be a bilinear operator, satisfying $\omega\circ (\alpha\otimes\alpha)=\alpha\circ \omega$. If  $$\Omega_t=\Omega+t\omega$$ is still a Hom-pre-Lie multiplication on $A$ for all $t$, we say that $\omega$ generates a {\bf (one-parameter) linear deformation} of a Hom-pre-Lie algebra $(A,\Omega,\alpha)$.
    \end{defi}
It is direct to check that $\Omega_t=\Omega+t\omega$ is a linear deformation of a Hom-pre-Lie algebra $(A,\Omega,\alpha)$ if and only if
 for any $x,y,z\in \g$, the following two equalities are satisfied:
\begin{eqnarray}
\nonumber\omega(x,y)\cdot \alpha(z)+\omega(x\cdot y,\alpha(z))-\alpha(x)\cdot \omega(y,z)-\omega(\alpha(x),y\cdot z) &&\\
\label{eq:4.1}  -\omega(y,x)\cdot \alpha(z)-\omega(y\cdot x,\alpha(z))+\alpha(y)\cdot \omega(x,z)+\omega(\alpha(y),x\cdot z)&=&0,\\
\label{eq:4.2}\omega(\omega(x,y),\alpha(z))-\omega(\alpha(x),\omega(y,z))-\omega(\omega(y,x),\alpha(z))+\omega(\alpha(y),\omega(x,z))&=&0.
\end{eqnarray}
Obviously, \eqref{eq:4.1} means that $\omega$ is a $2$-cocycle of the Hom-pre-Lie algebra $(A,\cdot,\alpha)$ with the coefficient in the regular representation, i.e. $$\partial_{\reg} \omega(\alpha(x),\alpha(y),\alpha(z))=0.$$ Furthermore, \eqref{eq:4.2} means that $(A,\omega,\alpha)$ is  a Hom-pre-Lie algebra.

\begin{defi}
Let $(A,\Omega,\alpha)$ be a Hom-pre-Lie algebra. Two linear deformations $\Omega_t^1=\Omega+t\omega_1$ and
$\Omega_t^2=\Omega+t\omega_2$ are said to be {\bf equivalent} if there exists a
linear operator $N\in \gl(A)$ such that $T_t={\Id}+tN$ is a
Hom-pre-Lie algebra homomorphism from $(A,\Omega_t^2,\alpha)$ to $(A,\Omega_t^1,\alpha)$. In particular, a
linear deformation $\Omega_t=\Omega+t\omega$ of a Hom-pre-Lie algebra $(A,\Omega,\alpha)$ is said to be {\bf trivial} if there exists
 a
linear operator $N\in \gl(A)$ such that $T_t={\Id}+tN$ is a
Hom-pre-Lie algebra homomorphism from  $(A,\Omega_t,\alpha)$ to  $(A,\Omega,\alpha)$.
\end{defi}

Let $T_t={\Id}+tN$ be a homomorphism from
$(A,\Omega_t^2,\alpha)$ to $(A,\Omega_t^1,\alpha)$. By \eqref{homo-2}, we have
\begin{equation}\label{equi-deformation-0}
N\circ\alpha=\alpha\circ N.
\end{equation}
Then by \eqref{homo-1} we obtain
\begin{eqnarray}
\label{equi-deformation-1}\omega_2(x,y)-\omega_1(x,y)=N(x)\cdot y+x\cdot N(y)-N(x\cdot y),\\
\label{equi-deformation-2}\omega_1(x,N(y))+\omega_1(N(x),y)=N(\omega_2(x,y))-N(x)\cdot N(y),\\
\label{equi-deformation-3}\omega_1(N(x),N(y))=0.
\end{eqnarray}
Note that \eqref{equi-deformation-1} means that $\omega_2-\omega_1=\partial_{\reg} (N\circ \alpha)$. Thus, we have

\begin{thm}\label{thm:iso3} Let $(A,\Omega,\alpha)$ be a Hom-pre-Lie algebra.
  If two linear deformations $\Omega_t^1=\Omega+t\omega_1$ and
$\Omega_t^2=\Omega+t\omega_2$ are equivalent, then $\omega_1$ and $\omega_2$ are in the same cohomology class of $H^2(A;A)$.
\end{thm}

\begin{defi}\label{defi:operator}
Let $(A,\cdot,\alpha)$ be a Hom-pre-Lie algebra. A linear operator $N\in \gl(A)$ is called a {\bf Nijenhuis operator} on $(A,\cdot,\alpha)$ if $N$ satisfies \eqref{equi-deformation-0} and the equation
      \begin{eqnarray}
        N(x)\cdot N(y)=N(x\cdot_N y),\quad \forall x,y\in A.
         \label{eq:Nij3}
        \end{eqnarray}
        where the product $\cdot_N$ is defined by
\begin{equation}\label{eq:5.2}
x\cdot_N y\triangleq  N(x)\cdot y+x\cdot N(y)-N(x\cdot y).
\end{equation}
\end{defi}

\begin{ex}{\rm
Let $(A,\cdot,\alpha)$ be the Hom-pre-Lie algebra given in Example \ref{ex-hom-pre-lie}. Furthermore, let
Let $N=\begin{pmatrix} -1 & 1 \\
                       0 & -1 \end{pmatrix}.$ More precisely, $N(e_1)=-e_1, N(e_2)=e_1-e_2.$ Then $N$ is a Nijenhuis operator on $(A,\cdot,\alpha)$.

}
\end{ex}

By \eqref{equi-deformation-0}-\eqref{equi-deformation-3}, a trivial linear deformation of a Hom-pre-Lie algebra gives rise to a Nijenhuis operator $N$. Conversely, a Nijenhuis operator can also generate a trivial linear deformation as the following theorem shows.

\begin{thm}\label{Nij-deformation}
Let N be a Nijenhuis operator on a Hom-pre-Lie algebra $(A,\Omega,\alpha)$. Then a linear deformation $(A,\Omega_t,\alpha)$ of a Hom-pre-Lie algebra $(A,\Omega,\alpha)$ can be obtained by putting
\begin{equation}\label{trivial-deformation-generate}
\omega(x,y)=\partial_{\reg}(N\circ \alpha)(x,y)=x\cdot_N y.
\end{equation}
Furthermore, this linear deformation $(A,\Omega_t,\alpha)$ is trivial.
\end{thm}
\begin{proof}
By \eqref{trivial-deformation-generate}, we have $\omega\circ(\alpha\otimes\alpha)=\alpha\circ\omega$ and $\partial_{\reg}\omega=0$. To show that $\omega$ generates a linear deformation of
the Hom-pre-Lie algebra $(A,\cdot,\alpha)$, we only need to verify that~\eqref{eq:4.2} holds, that is, $(A,\omega,\alpha)$ is  a Hom-pre-Lie algebra. By straightforward computations, we have
\begin{eqnarray*}&&\omega(\omega(x,y),\alpha(z))-\omega(\alpha(x),\omega(y,z))-\omega(\omega(y,x),\alpha(z))+\omega(\alpha(y),\omega(x,z))\\
&=&N(N(x)\cdot y)\cdot \alpha(z)+N(x\cdot N(y))\cdot \alpha(z)-N^2(x\cdot y)\cdot \alpha(z)\\
&&+(N(x)\cdot y)\cdot N(\alpha(z))+(x\cdot N(y))\cdot N(\alpha(z))-N(x\cdot y)\cdot N(\alpha(z))\\
&&-N\big((N(x)\cdot y)\cdot \alpha(z)\big)-N\big((x\cdot N(y))\cdot \alpha(z)\big)+N\big(N(x\cdot y)\cdot \alpha(z)\big)\\
&&-N(\alpha(x))\cdot (N(y)\cdot z)-N(\alpha(x))\cdot (y\cdot N(z))+N(\alpha(x))\cdot N(y\cdot z)\\
&&-\alpha(x)\cdot N(N(y)\cdot z)-\alpha(x)\cdot N(y\cdot N(z))+\alpha(x)\cdot N^2(y\cdot z)\\
&&+N\big(\alpha(x)\cdot (N(y)\cdot z)\big)+N\big(\alpha(x)\cdot (y\cdot N(z))\big)-N\big(\alpha(x)\cdot N(y\cdot z)\big)\\
&&-N(N(y)\cdot x)\cdot \alpha(z)-N(y\cdot N(x))\cdot \alpha(z)+N^2(y\cdot x)\cdot \alpha(z)\\
&&-(N(y)\cdot x)\cdot N(\alpha(z))-(y\cdot N(x))\cdot N(\alpha(z))+N(y\cdot x)\cdot N(\alpha(z))\\
&&+N\big((N(y)\cdot x)\cdot \alpha(z)\big)+N\big((y\cdot N(x))\cdot \alpha(z)\big)-N\big(N(y\cdot x)\cdot \alpha(z)\big)\\
&&+N(\alpha(y))\cdot (N(x)\cdot z)+N(\alpha(y))\cdot (x\cdot N(z))-N(\alpha(y))\cdot N(x\cdot z)\\
&&+\alpha(y)\cdot N(N(x)\cdot z)+\alpha(y)\cdot N(x\cdot N(z))-\alpha(y)\cdot N^2(x\cdot z)\\
&&-N\big(\alpha(y)\cdot (N(x)\cdot z)\big)-N\big(\alpha(y)\cdot (x\cdot N(z))\big)+N\big(\alpha(y)\cdot N(x\cdot z)\big).
\end{eqnarray*}
Since $N$ commutes with $\alpha$, by the definition of a Hom-pre-Lie algebra, we have
\begin{eqnarray*}
 &&(N(x)\cdot y)\cdot N(\alpha(z))-N(\alpha(x))\cdot (y\cdot N(z))
 -(y\cdot N(x))\cdot N(\alpha(z))+\alpha(y)\cdot (N(x)\cdot N(z))=0,\\
 &&(N(y)\cdot x)\cdot N(\alpha(z))-N(\alpha(y))\cdot (x\cdot N(z))
 -(x\cdot N(y))\cdot N(\alpha(z))+\alpha(x)\cdot (N(y)\cdot N(z))=0,\\
 &&(N(x)\cdot N(y))\cdot \alpha(z)-N(\alpha(x))\cdot (N(y)\cdot z)
-(N(y)\cdot N(x))\cdot \alpha(z)+N(\alpha(y))\cdot (N(x)\cdot z)=0.
\end{eqnarray*}
Since $N$ is a Nijenhuis operator, by the definition of a Hom-pre-Lie algebra, we have
\begin{eqnarray*}&&-N(x\cdot y)\cdot N(\alpha(z))+N(\alpha(x))\cdot N(y\cdot z)+N(y\cdot x)\cdot N(\alpha(z))-N(\alpha(y))\cdot N(x\cdot z)\\
&=&-N\big((x\cdot y)\cdot_N \alpha(z)\big)+N\big(\alpha(x)\cdot_N (y\cdot z)\big)
+N\big((y\cdot x)\cdot_N \alpha(z)\big)-N\big(\alpha(y)\cdot_N (x\cdot z)\big)\\
&=&-N\big(N(x\cdot y)\cdot \alpha(z)\big)-N\big((x\cdot y)\cdot N(\alpha(z))\big)
+N\big(N(\alpha(x))\cdot (y\cdot z)\big)+N\big(\alpha(x)\cdot N(y\cdot z)\big)\\
&&+N\big(N(y\cdot x)\cdot \alpha(z)\big)+N\big((y\cdot x)\cdot N(\alpha(z))\big)
-N\big(N(\alpha(y))\cdot (x\cdot z)\big)-N\big(\alpha(y)\cdot N(x\cdot z)\big).
\end{eqnarray*}
Therefore, we have
\begin{eqnarray*}&&\omega(\omega(x,y),\alpha(z))-\omega(\alpha(x),\omega(y,z))-\omega(\omega(y,x),\alpha(z))+\omega(\alpha(y),\omega(x,z))\\
&=&-N\big((x\cdot y)\cdot N(\alpha(z))\big)+N\big(\alpha(x)\cdot (y\cdot N(z))\big)
+N\big((y\cdot x)\cdot N(\alpha(z))\big)-N\big(\alpha(y)\cdot (x\cdot N(z))\big)\\
&&-N\big((N(x)\cdot y)\cdot \alpha(z)\big)+N\big(N(\alpha(x))\cdot (y\cdot z)\big)
+N\big((y\cdot N(x))\cdot \alpha(z)\big)-N\big(\alpha(y)\cdot (N(x)\cdot z)\big)\\
&&+N\big((N(y)\cdot x)\cdot \alpha(z)\big)-N\big(N(\alpha(y))\cdot (x\cdot z)\big)
-N\big((x\cdot N(y))\cdot \alpha(z)\big)+N\big(\alpha(x)\cdot (N(y)\cdot z)\big)\\
&=&0.
\end{eqnarray*}
Thus, $\omega$ generates a linear deformation of the Hom-pre-Lie algebra $(A,\cdot_A,\alpha)$.

We define $T_t={\Id}+tN$. It is straightforward to deduce that $T_t={\Id}+tN$ is a Hom-pre-Lie algebra homomorphism from $(A,\Omega_t,\alpha)$ to  $(A,\Omega,\alpha)$. Thus, the deformation generated by $\omega=\partial_{\reg}(N\circ \alpha)$ is trivial.
\end{proof}

By \eqref{eq:Nij3} and Theorem \ref{Nij-deformation}, we have the following corollary.
\begin{cor}
Let $N$ be a Hom-Nijenhuis operator on a Hom-pre-Lie algebra $(A,\Omega,\alpha)$, then $(A,\cdot_N,\alpha)$ is a Hom-pre-Lie algebra, and $N$ is a homomorphism from $(A,\cdot_N,\alpha)$ to
$(A,\cdot,\alpha)$.
\end{cor}
At the end of this section, we recall linear deformations of Hom-Lie algebras and Nijenhuis operators on Hom-Lie algebras, which give trivial deformations of Hom-Lie algebras.
\begin{defi}{\rm(\cite{sheng3})}\label{defi:deformation}
Let $(\g,[\cdot,\cdot],\alpha)$ be a Hom-Lie algebra and $\omega\in \huaC^2(\g;\g)$ a skew-symmetric bilinear operator, satisfying $\omega\circ (\alpha\otimes\alpha)=\alpha\circ \omega$. Consider a $t$-parameterized family of bilinear operations, $$[\cdot,\cdot]_t=[\cdot,\cdot]+t\omega.$$ If $(\g,[\cdot,\cdot]_t,\alpha)$ is a Hom-Lie algebra for all $t$, we say that $\omega$ generates a {\bf (one-parameter) linear deformation} of a Hom-Lie algebra $(\g,[\cdot,\cdot],\alpha)$.
    \end{defi}
\begin{defi}{\rm(\cite[Definition 6.6]{sheng3})}
Let $(\g,[\cdot,\cdot],\alpha)$ be a Hom-Lie algebra. A linear operator $N\in \gl(A)$ is called a {\bf Nijenhuis operator} on $(\g,[\cdot,\cdot],\alpha)$ if we have
\begin{eqnarray}
\label{homo-3}N\circ \alpha&=&\alpha\circ N,\\
\label{homo-4}[N(x),N(y)]&=&N[x,y]_N,\quad \forall x,y\in \g.
\end{eqnarray}
where the bracket $[\cdot,\cdot]_N$ is defined by
\begin{equation}
[x,y]_N\triangleq  [N(x),y]+[x,N(y)]-N[x,y].
\end{equation}
    \end{defi}
\begin{pro}
If $\omega \in C^2(A;A)$ generates a linear deformation of a Hom-pre-Lie algebra $(A,\cdot,\alpha)$, then $\omega_C \in \huaC^2(A^C;A)$ defined by$$\omega_C(x,y)=\omega(x,y)-\omega(y,x)$$ generates a linear deformation of the sub-adjacent Hom-Lie algebra $A^C$.
\end{pro}

\begin{proof}
Assume that $\omega$ generates a linear deformation of a Hom-pre-Lie algebra $(A,\cdot,\alpha)$. Then $(A,\Omega_t,\alpha)$ is a Hom-pre-Lie algebra. Consider its corresponding sub-adjacent Hom-Lie algebra $(A,[\cdot,\cdot]_t,\alpha)$, we have
\begin{eqnarray*}
  [x,y]_t&=&\Omega_t(x,y)-\Omega_t(y,x)\\
  &=&x\cdot y+t\omega(x,y)-y\cdot x-t\omega(y,x)\\
  &=&[x,y]_C+t\omega_C(x,y).
\end{eqnarray*}
Thus, $\omega_C$ generates a linear deformation of $A^C$.
\end{proof}

\begin{pro}
If $N$ is a Nijenhuis operator on a Hom-pre-Lie algebra $(A,\cdot,\alpha)$, then $N$ is a Nijenhuis operator on the sub-adjacent Hom-Lie algebra $A^C$.
\end{pro}
\begin{proof}
Since $N$ is a Nijenhuis operator on $(A,\cdot,\alpha)$, we have $N\circ \alpha=\alpha\circ N$. For all $x,y\in A$, by Definition \ref{defi:operator}, we have
\begin{eqnarray*}
  [N(x),N(y)]_C&=&N(x)\cdot N(y)-N(y)\cdot N(x)\\
  &=&N\big(N(x)\cdot y+x\cdot N(y)-N(x\cdot y)-N(y)\cdot x-y\cdot N(x)+N(y\cdot x)\big)\\
  &=&N([N(x),y]_C+[x,N(y)]_C-N[x,y]_C).
\end{eqnarray*}
Thus, $N$ is a Nijenhuis operator on $A^C$.
\end{proof}

\section{$\huaO$-operators and Hessian structures on Hom-pre-Lie algebras}

In this section, we introduce the notion of an $\huaO$-operator and a Hessian structure on a Hom-pre-Lie algebra, and study their relations.

\begin{defi}
Let $(A,\cdot,\alpha)$ be a Hom-pre-Lie algebra and $(V,\beta,\rho,\mu)$ a representation. A linear map $T:V\longrightarrow A$ is called an {\bf $\huaO$-operator} on $(A,\cdot,\alpha)$ if for all $u,v\in V$,
\begin{eqnarray}
  \label{eq:10.1}T\circ\beta&=&\alpha\circ T,\\
   \label{eq:10.2}T(u)\cdot T(v)&=&T\big(\rho(T(u))(v)+\mu(T(v))(u)\big).
\end{eqnarray}
\end{defi}
\begin{pro}
Let $(A,\cdot,\alpha)$ be a Hom-pre-Lie algebra and $(V,\beta,\rho,\mu)$ a representation. A linear map $T:V\longrightarrow A$ is an $\huaO$-operator on $(A,\cdot,\alpha)$ if and only if\\ $$\bar{T}=\begin{pmatrix} 0 & T \\ 0 & 0 \end{pmatrix}:A\oplus V\longrightarrow A\oplus V$$\\
is a Nijenhuis operator on the semidirect product Hom-pre-Lie algebra $A\ltimes_{(\rho,\mu)}V$.
\end{pro}
\begin{proof}
First it is obvious that $\bar{T}\circ (\alpha+\beta)=(\alpha+\beta)\circ \bar{T}$ if and only if $T\circ\beta=\alpha\circ T$. Then for all $x,y\in A, u,v\in V$, we have
\begin{equation}
\bar{T}(x+u)\cdot \bar{T}(y+v)=T(u)\cdot T(v).
\end{equation}
On the other hand, since $\bar{T}^2=0$, we have
\begin{eqnarray*}&&\bar{T}((x+u)\cdot_{\bar{T}} (y+v))\\
&=&\bar{T}\big(\bar{T}(x+u)\cdot_\ltimes(y+v)+(x+u)\cdot_\ltimes \bar{T}(y+v)-\bar{T}((x+u)\cdot_\ltimes (y+v))\big)\\
&=&\bar{T}\big(T(u)\cdot_\ltimes(y+v)+(x+u)\cdot_\ltimes T(v)\big)\\
&=&T\big(\rho(T(u))(v)+\mu(T(v))(u)\big),
\end{eqnarray*}
which implies that  $T$ is an $\huaO$-operator on $(A,\cdot,\alpha)$ if and only if $\bar{T}$ is a Nijenhuis operator on the semidirect product Hom-pre-Lie algebra $A\ltimes_{(\rho,\mu)}V$.
\end{proof}
\begin{defi}
A {\bf Hessian structure} on a Hom-pre-Lie algebra $(A,\cdot,\alpha)$ is a symmetric nondegenerate $2$-cocycle $\huaB \in Sym^2(A^\ast)$, i.e. $\partial_T\huaB=0$, satisfying $\huaB \circ (\alpha\otimes\alpha)=\huaB$.  More precisely,
\begin{eqnarray}
  \label{hom-hessian-1}\huaB(\alpha(x),\alpha(y))&=&\huaB(x,y),\\
  \label{hom-hessian-2}\huaB(x\cdot y,\alpha(z))-\huaB(\alpha(x),y\cdot z)&=&\huaB(y\cdot x,\alpha(z))-\huaB(\alpha(y),x\cdot z),\quad \forall x,y,z\in A.
\end{eqnarray}

\end{defi}
Let $A$ be a vector space, for all $\huaB \in Sym^2(A^\ast)$, the linear map $\huaB^\sharp:A \longrightarrow A^\ast$ is given by
\begin{equation}\label{hessian}
\langle \huaB^\sharp(x),y\rangle=\huaB(x,y),\quad \forall x,y\in A.
\end{equation}
\begin{pro}
With the above notations, $\huaB \in Sym^2(A^\ast)$ is a Hessian structure on a Hom-pre-Lie algebra $(A,\cdot,\alpha)$ if and only if $(\huaB^\sharp)^{-1}$ is an $\huaO$-operator on $(A,\cdot,\alpha)$ with respect to the representation $(A^*,(\alpha^{-1})^*, L^\star-R^\star,-R^\star)$.
\end{pro}
\begin{proof}
Let $\huaB$ be a Hessian structure on a Hom-pre-Lie algebra $(A,\cdot,\alpha)$.
For all $x,y\in A$, by \eqref{hom-hessian-1} and \eqref{hessian}, we have
\begin{eqnarray*}\langle\huaB^\sharp(\alpha(x))-(\alpha^{-1})^\ast(\huaB^\sharp(x)),\alpha(y)\rangle
&=&\langle\huaB^\sharp(\alpha(x)),\alpha(y)\rangle-\langle(\alpha^{-1})^\ast(\huaB^\sharp(x)),\alpha(y)\rangle\\
&=&\langle\huaB^\sharp(\alpha(x)),\alpha(y)\rangle-\langle\huaB^\sharp(x),y\rangle\\
&=&\huaB(\alpha(x),\alpha(y))-\huaB(x,y)\\
&=&0,
\end{eqnarray*}
which implies that $\huaB^\sharp\circ \alpha=(\alpha^{-1})^\ast\circ \huaB^\sharp$. Thus, we obtain that
\begin{equation}\label{hessian-operator1}
(\huaB^\sharp)^{-1}\circ (\alpha^{-1})^\ast=\alpha\circ (\huaB^\sharp)^{-1}.
\end{equation}

For all $\xi,\eta,\gamma \in A^\ast$, set $x=\alpha^{-1}((B^\sharp)^{-1}(\xi))$, $y=\alpha^{-1}((B^\sharp)^{-1}(\eta))$, $z=\alpha^{-1}((B^\sharp)^{-1}(\gamma))$, i.e. $\huaB^\sharp(\alpha(x))=\xi$, $\huaB^\sharp(\alpha(y))=\eta$, $\huaB^\sharp(\alpha(z))=\gamma$, by \eqref{hom-hessian-2} and \eqref{hessian}, we have
\begin{eqnarray*}&&\langle (\huaB^\sharp)^{-1}(\xi)\cdot (\huaB^\sharp)^{-1}(\eta)-(\huaB^\sharp)^{-1}\big((L^\star-R^\star)((\huaB^\sharp)^{-1}(\xi))(\eta)-R^\star( (\huaB^\sharp)^{-1}(\eta))(\xi)\big),(\alpha^{-1})^\ast(\gamma) \rangle\\
&=&\langle \alpha(x)\cdot \alpha(y),(\alpha^{-1})^\ast(\gamma) \rangle -\langle(\huaB^\sharp)^{-1}\big((L^\star-R^\star)(\alpha(x))(\eta)-R^\star(\alpha(y))(\xi)\big),\huaB^\sharp(\alpha^2(z))\rangle \\
&=&\langle x\cdot y,\gamma \rangle -\langle(\huaB^\sharp)^\ast(\huaB^\sharp)^{-1}\big((L^\star-R^\star)(\alpha(x))(\eta)-R^\star(\alpha(y))(\xi)\big), \alpha^2(z)\rangle\\
&=&\langle x\cdot y,\huaB^\sharp(\alpha(z)) \rangle -\langle(L^\star-R^\star)(\alpha(x))(\eta)-R^\star(\alpha(y))(\xi), \alpha^2(z)\rangle\\
&=&\langle  x\cdot y,\huaB^\sharp(\alpha(z))  \rangle +\langle \eta,(L-R)(x)(z)\rangle-\langle\xi,R(y)(z)\rangle\\
&=& \langle x\cdot y,\huaB^\sharp(\alpha(z)) \rangle +\langle\huaB^\sharp(\alpha(y)),x\cdot z\rangle
-\langle\huaB^\sharp(\alpha(y)),z\cdot x\rangle-\langle \huaB^\sharp(\alpha(x)), z \cdot y \rangle\\
&=&\huaB(\alpha(z),x\cdot y)+B(\alpha(y),x\cdot z)-B(\alpha(y),z\cdot x)-B(\alpha(x),z\cdot y)\\
&=&\huaB(\alpha(z),x\cdot y)+B(x\cdot z,\alpha(y))-B(z\cdot x,\alpha(y))-B(\alpha(x),z\cdot y)\\
&=&0,
\end{eqnarray*}
which implies that
\begin{equation}\label{hessian-operator2}
(\huaB^\sharp)^{-1}(\xi)\cdot (\huaB^\sharp)^{-1}(\eta)=(\huaB^\sharp)^{-1}\big((L^\star-R^\star)((\huaB^\sharp)^{-1}(\xi))(\eta)-R^\star( (\huaB^\sharp)^{-1}(\eta))(\xi)\big).
\end{equation}
By \eqref{hessian-operator1} and \eqref{hessian-operator2}, we deduce that $(\huaB^\sharp)^{-1}$ is an $\huaO$-operator on $(A,\cdot,\alpha)$ with respect to the representation $(A^*,(\alpha^{-1})^*, L^\star-R^\star,-R^\star)$.

The converse part can be proved similarly. We omit details. The
 proof is finished.
\end{proof}

\end{document}